\documentclass{article}

\usepackage{hyperref}
\usepackage[totalheight=21.8 true cm, totalwidth=15 true cm]{geometry}
\usepackage[shortcuts]{extdash}

\usepackage{xypic}

\usepackage{enumerate}
\usepackage{amsthm}
\usepackage{amsmath}
\usepackage{amssymb}
\usepackage{amsfonts}
\usepackage{dsfont}

\newtheorem{theo}{Theorem}[section]
\newtheorem{lemma}[theo]{Lemma}
\newtheorem{prop}[theo]{Proposition}
\newtheorem{cor}[theo]{Corollary}
\newtheorem{defi}[theo]{Definition}

\newtheorem*{theononumber}{Theorem}
\newtheorem*{conjmatzat}{Matzat's Conjecture}

\theoremstyle{definition}

\newcommand{\f}{\phi}

\newcommand{\G}{\mathcal{G}}

\newcommand{\Gl}{\operatorname{GL}}

\newcommand{\Aut}{\underline{\operatorname{Aut}}}
\newcommand{\Hom}{\operatorname{Hom}}

\newcommand{\id}{\operatorname{id}}

\newcommand{\N}{\mathcal{N}}

\newcommand{\de}{\delta}

\def\Sl{\operatorname{SL}}

\newcommand{\kb}{\overline{k}}

\newcommand{\op}{\operatorname{op}}

\title{Subgroups of free proalgebraic groups and Matzat's conjecture for function fields}
\author{Michael Wibmer\thanks{This work was supported by the NSF grants DMS-1760212, DMS-1760413, DMS-1760448 and the Lise Meitner grant M-2582-N32 of the Austrian Science Fund FWF.}}

\date{\today}

\begin{document}
\maketitle

\begin{abstract}	
	We show that a closed finite index subgroup of a free proalgebraic group  is itself a free proalgebraic group. Our main motivation for this result is an application in differential Galois theory: The absolute differential Galois group of a one-variable function field is a free proalgebraic group.
\let\thefootnote\relax\footnotetext{

	Michael Wibmer, Institute of Analysis and Number Theory, Graz University of Technology, Kopernikusgasse 24, 8010 Graz, Austria, \texttt{wibmer@math.tugraz.at}
}

\end{abstract}

\section*{Introduction}

Understanding the absolute differential Galois group of interesting differential fields is a central problem in differential Galois theory. Matzat's conjecture addresses this question for one-variable function fields.

\begin{conjmatzat}
	Let $k$ be an algebraically closed field of characteristic zero and let $K$ be a one-variable function field over $k$, equipped with a non-trivial $k$-derivation. Then the absolute differential Galois group of $K$ is the free proalgebraic group on a set of cardinality $|K|$.
\end{conjmatzat}
The first special case of Matzat's conjecture was proved in \cite{BachmayrHarbaterHartmannWibmer:FreeDifferentialGaloisGroups}. There it was  shown that the conjecture holds for $K=k(x)$, the rational function field over $k$ equipped with the standard derivation $\frac{d}{dx}$, provided that $k$ is countable and of infinite transcendence degree (over $\mathbb{Q}$). Subsequently, in \cite{BachmayrHarbaterHartmannWibmer:TheDifferentialGaloisGroupOfRationalFunctionField}, Matzat's conjecture was proved for $(k(x), \frac{d}{dx})$ provided that $k$ is of infinite transcendence degree. The main result of our article is the following fairly general case of Matzat's conjecture:

\begin{theononumber}[Theorem \ref{theo: main matzat}]
	Let $k$ be an algebraically closed field of characteristic zero of infinite transcendence degree and let $K$ be a one-variable function field over $k$, equipped with a non-trivial $k$-derivation. Then the absolute differential Galois group of $K$ is the free proalgebraic group on a set of cardinality $|K|$.
\end{theononumber}
In other words, Matzat's conjecture holds, provided that $k$ has infinite transcendence degree. In fact, we show that, for a fixed algebraically closed field $k$ of characteristic zero (of arbitrary transcendence degree), Matzat's conjecture for $(k(x), \frac{d}{dx})$ implies the general form of Matzat's conjecture over $k$. Our proof of this result has two main steps: 
\begin{enumerate}
	\item We show that the absolute differential Galois group of a one-variable function field over $k$, equipped with a non-zero $k$-derivation, is isomorphic to a closed finite index subgroup of the absolute differential Galois group of $(k(x), \frac{d}{dx})$.
	\item We show that closed finite index subgroups of free proalgebraic groups are themselves free.
\end{enumerate}
The second point can be seen as an algebraic-geometric analog of the Nielsen-Schreier theorem, stating that subgroups of (abstract) free groups are themselves free. It can also be seen as an analog of the corresponding result for free profinite groups. A closed subgroup of a free profinite group is in general not free. However, an open (=closed of finite index) subgroup of a free profinite group is free (see \cite[Section 17.6]{FriedJarden:FieldArithmetic} or \cite[Section 3.6]{RibesZalesskii:ProfiniteGroups}). Our proof of (ii) follows along the lines of \cite{RibesSteinberg:AWreathProductApproachToClassicalSubgroupTheorems} which provides an algebraic approach to the Nielsen-Schreier theorem based on wreath products.

We conclude the introduction with an outline of the paper. In the first section we show that a closed finite index subgroup of a free proalgebraic group is itself free. In the second section we apply this result to obtain a proof of Matzat's conjecture for function fields over constants of infinite transcendence degree.

\section{Subgroups of free proalgebraic groups}


\subsection{Notation and quotients}

Throughout this article $k$ is a field (our base field) and $\kb$ is an algebraic closure of $k$. All schemes are assumed to be over $k$ unless the contrary is indicated. For brevity and to  emphasize the analogy with the theory of profinite groups, we use the term ``proalgebraic group (over $k$)'' in lieu of ``affine group scheme (over $k$)''. Similarly, an algebraic group (over $k$) is an affine group scheme of finite type over $k$. A closed subgroup of a proalgebraic group is, by definition, a closed subgroup scheme. We will usually identify a proalgebraic group $G$ with its functor of points $R\rightsquigarrow G(R)$ from the category of $k$-algebras to the category of groups.
For an affine scheme $X$, we denote the ring of global sections of $X$ by $k[X]$, i.e., $X\simeq \Hom(k[X],-)$  as functors from the category of $k$-algebras to the category of groups.

A morphism $\f\colon G\to H$ of proalgebraic groups is a \emph{closed embedding} if it induces an isomorphism between $G$ and a closed subgroup of $H$; equivalently, $\f\colon G(R)\to H(R)$ is injective for every $k$-algebra $R$ (\cite[Theorem 15.3]{Waterhouse:IntroductiontoAffineGroupSchemes}).
For a morphism $\f\colon G\to H$ of proalgebraic groups we denote with $\f(G)$ the scheme theoretic image of $\f$. It is the closed subgroup of $H$ defined by the kernel of the dual map $\f^*\colon k[H]\to k[G]$.

Let $H$ be a closed subgroup of a proalgebraic group $G$. We will need to work with the left and right coset spaces $G/H$ and $H\backslash G$ of $H$ in $G$.
So let us make precise what we mean by that: We define 
$G/H$ as the sheaf in the fpqc topology associated to the functor $R\rightsquigarrow G(R)/H(R)$ from the category of $k$-algebras to the category of groups, where $G(R)/H(R)$ is the usual coset space of $H(R)$ in $G(R)$ (as in \cite[Chapter III, \S 3, Section 7]{DemazureGabriel:GroupesAlgebriques}). The functor $H\backslash G$ is defined similarly. We note that if the sheaf $\tilde{G/H}$ in the fppf topology associated to the functor $R\rightsquigarrow G(R)/H(R)$ is a scheme (this, e.g., is the case if $G$ is algebraic (\cite[Chapter III, \S 3, Theorem 5.4]{DemazureGabriel:GroupesAlgebriques})), then $G/H=\tilde{G/H}$.
In general, $G/H$ need not be a scheme and if $G/H$ is a scheme it need not be affine. However, we are mostly interested in the following two situations:
\begin{itemize}
	\item The closed subgroup $H$ is normal in $G$. In this case $G/H$ is an affine scheme and therefore a proalgebraic group (\cite[Chapter III, \S 3, Theorem 7.2]{DemazureGabriel:GroupesAlgebriques}).  
	\item The functor $H\backslash G$ is a finite (and hence affine) $k$-scheme. 
\end{itemize}
Note that $H\backslash G$ comes equipped with a right action of $G$ and a canonical morphism $G\to H\backslash G$, that is $G$-equivariant for the right multiplication of $G$ on $G$.
The inversion $g\mapsto g^{-1}$ induces an isomorphism $G/H\to H\backslash G$. In particular, $G/H$ is a finite $k$-scheme if and only if $H\backslash G$ is a finite $k$-scheme and the $k$-dimensions of $k[G/H]$ and $k[H\backslash G]$ agree. If this is the case, we say that $H$ has \emph{finite index} in $G$, or that $H$ is a finite index subgroup of $G$. The index $[G:H]$ of $H$ in $G$ is defined as the dimension of $k[G/H]$ as a $k$-vector space. 

A morphism $\f\colon G\to H$ of proalgebraic groups is a \emph{quotient map} if it induces an isomorphism between $H$ and $G/\ker(\f)$. This is equivalent to the dual map $\f^*\colon k[H]\to k[G]$ being injective. For brevity, we call a closed normal subgroup $N$ of a proalgebraic group $G$ a \emph{coalgebraic subgroup} of $G$ if $G/N$ is an algebraic group.
For a scheme $X$ over $k$ and a $k$-algebra $R$, we denote with $X_R$ the $R$-scheme obtain from $X$ by base change via $k\to R$.

For a scheme $\Sigma$ over $k$, we define $\Aut(\Sigma)$ as the functor from the category of $k$-algebras to the category of groups given by $\Aut(\Sigma)(R)=\operatorname{Aut}(\Sigma_R)$.
If $\Sigma$ is a finite $k$-scheme, $\Aut(\Sigma)$ is an algebraic group (\cite[Chapter II, \S1, 2.7]{DemazureGabriel:GroupesAlgebriques}).

\begin{lemma} \label{lemma: exists coalgebraic subgroup is finite index subgroup}
	Let $H$ be a closed finite index subgroup of a proalgebraic group $G$. Then there exists a coalgebraic subgroup $N$ of $G$ with $N\leq H$.
\end{lemma}	
\begin{proof}
	Set $\Sigma=G/H$ and let $N$ be the kernel of the morphism $G\to \Aut(\Sigma)$ of proalgebraic groups corresponding to the (left) $G$-action on $\Sigma$. Then $N\leq H$ and because $G/N$ embeds into the algebraic group $\Aut(\Sigma)$, also $G/N$ is algebraic.
\end{proof}

The following lemma summarizes what we will need to know about quotients.
\begin{lemma} \label{lemma: quotients}
	Let $H$ be a closed finite index subgroup of a proalgebraic group $G$.

	\begin{enumerate}
		\item For a $k$-algebra $R$ and $g_1,g_2\in G(R)$ we have $H(R)g_1=H(R)g_2$ if and only if $g_1$ and $g_2$ have the same image under the canonical map $G\to H\backslash G$.  
		\item If $k$ is algebraically closed, the map $G(k)\to (H\backslash G)(k)$ is surjective.
		\item If $G$ is reduced, then also $H\backslash G$ is reduced.
	\end{enumerate}

\end{lemma}
\begin{proof}
	Point (i) follows as in \cite[Chapter III, \S 1, Example 2.4]{DemazureGabriel:GroupesAlgebriques} by using the equivalence relation on $G$ defined by $H$.
	
	For (ii), let $N$ be a coalgebraic subgroup of $G$ such that $N\leq H$ (Lemma \ref{lemma: exists coalgebraic subgroup is finite index subgroup}). Then $H/N$ can be identified with a closed subgroup of $G/N$ and we have a commutative diagram	
	$$
	\xymatrix{
	G \ar[r] \ar[d] & G/N \ar[d] \\
	H\backslash G \ar[r] & (H/N)\backslash(G/N)	
	}
$$
in which the lower horizontal arrow is an isomorphism. It thus suffices to show that the upper horizontal arrow and the right vertical arrow are surjective on $k$-points. For $G\to G/N$ this is \cite[Chapter III, \S 3, Cor. 7.6]{DemazureGabriel:GroupesAlgebriques}. For the right vertical arrow this is the claim of (ii) under the additional assumption that $G$ is algebraic. But if $G$ is algebraic, then $G\to H\backslash G$ is a faithfully flat morphism of affine schemes of finite type over $k$ (\cite[Chapter III, \S 3, Prop. 2.5]{DemazureGabriel:GroupesAlgebriques}) and therefore surjective on the $k$-points.

For (iii), we can reduce to the case that $G$ is algebraic as in (ii). (Note that $G/N$ is reduced if $G$ is reduced, because $k[G/N]\to k[G]$ is injective.) Then $G\to H\backslash G$ is faithfully flat and so the dual map $k[H\backslash G]\to k[G]$ is injective. So if $k[G]$ is reduced, also $k[H\backslash G]$ is reduced.
\end{proof}

\subsection{Wreath products for proalgebraic groups}

Our proof that closed finite index subgroups of free proalgebraic groups are themselves free, follows along the lines of \cite{RibesSteinberg:AWreathProductApproachToClassicalSubgroupTheorems}, which relies on properties of wreath products.
In this section we introduce wreath products for proalgebraic groups and establish some of their basic properties. The author was not able to locate any reference that deals with wreath products of (pro)algebraic groups. Our definitions and results are mostly straight forward adaptions from the classical case of groups. However, a question of representability intervenes.

Let us first recall the construction of wreath products in the category of groups. Let $A$ be a group and let $G$ be a group acting on a set $\Sigma$ from the right. The wreath product $A\wr G$ of $A$ by $G$ is the semidirect product
$$ A\wr G=A^\Sigma\rtimes G,$$
where $A^\Sigma$ is the set of all maps from $\Sigma$ to $A$ considered as a group under pointwise multiplication and $G$ acts on $A^\Sigma$ (from the left) via $g(f)(s)=f(sg)$ for $g\in G,\ f\in A^\Sigma$ and $s\in \Sigma$.

\medskip

Now let $A$ and $G$ be proalgebraic groups over $k$ and let $\Sigma$ be a $k$-scheme equipped with a right action of $G$. Let $A^\Sigma$ denote the functor from the category of $k$-algebras to the category of groups given by $A^\Sigma(R)=\Hom(\Sigma_R,A_R)$ for any $k$-algebra $R$. The group structure on the set $\Hom(\Sigma_R,A_R)$ of all $R$-scheme morphisms from $\Sigma_R$ to $A_R$ is given by $f_1f_2\colon \Sigma_R\xrightarrow{(f_1,f_2)} A_R\times_R A_R\to A_R$ for $f_1,f_2\in \Hom(\Sigma_R,A_R)$, where the second map is the multiplication in $A_R$.
In general, there is no reason why $A^\Sigma$ should be a scheme (i.e., representable), however: 

\begin{lemma} \label{lemma: AhochSigma}
	If $\Sigma$ is a finite $k$-scheme, then $A^\Sigma$ is a proalgebraic group. Moreover, the diagonal morphism $d\colon A\to A^\Sigma$ is a closed embedding of proalgebraic groups.
\end{lemma}
\begin{proof}
	For a $k$-algebra $R$, we have 
	\begin{align*}
	A^\Sigma(R)&=\Hom(\Sigma_R,A_R)=\Hom_R(R\otimes_kk[A], R\otimes_kk[\Sigma])=\Hom_k(k[A],R\otimes_k k[\Sigma])=\\ &=\Hom_{k[\Sigma]}(k[A]\otimes_k k[\Sigma],R\otimes_k k[\Sigma])=A_{k[\Sigma]}(R\otimes_k k[\Sigma])=(\Pi_{k[\Sigma]/k}A_{k[\Sigma]})(R),
	\end{align*}
	where $\Pi_{k[\Sigma]/k}X$ denotes the Weil restriction $R\rightsquigarrow X(R\otimes_k k[\Sigma])$ of an affine $k[\Sigma]$-scheme $X$ to $k$. The Weil restriction $\Pi_{k[\Sigma]/k}X$ is an affine $k$-scheme if $k[\Sigma]$ is a finite dimensional $k$-vector space (\cite[Chapter I, \S1, Prop. 6.6]{DemazureGabriel:GroupesAlgebriques}). Therefore $A^\Sigma$ is an affine scheme.

	The diagonal morphism $d\colon A\to A^\Sigma$ is defined as follows: For a $k$-algebra $R$ and $a\in A(R)$ the morphism $d(a)\colon \Sigma_R\to A_R$ is given by $d(a)(s)=a$ for $s\in \Sigma(R')$ and any $R$-algebra $R'$. For every $k$-algebra $R$, the map $A(R)\to (A^\Sigma)(R),\ a\mapsto d(a)$ is an injective group homomorphism. Thus $d$ is a closed embedding.
\end{proof}

Henceforth we assume that $\Sigma$ is a finite $k$-scheme.\footnote{This is sufficient for our purposes, since in the application to free proalgebraic groups (Section \ref{subsec: Finite index subgroups of free proalgebraic groups are free}), the role of $\Sigma$ will be played by $H\backslash \Gamma$, where $H$ is a closed finite index subgroup of a free proalgebraic group $\Gamma$.}
The proalgebraic group $G$ acts on the proalgebraic group $A^\Sigma$ from the left by $g(f)\colon \Sigma_R\xrightarrow{g}\Sigma_R\xrightarrow{f}A_R$, i.e., for a $k$-algebra $R$, $g\in G(R)$ and $f\in A^\Sigma(R)$ we have $g(f)(s)=f(sg)$, where $s\in \Sigma(R')$ and $R'$ is an $R$-algebra. We can therefore form the semidirect product
$$ A\wr G=A^\Sigma\rtimes G,$$
henceforth called the \emph{wreath product of $A$ by $G$} (with respect to the $G$-scheme $\Sigma$).

For a morphism $\f\colon A\to B$ of proalgebraic groups, we define $\f^\Sigma\colon A^\Sigma\to B^\Sigma$ by $\f^\Sigma(f)(s)=\f(f(s))$ for $f\in A^\Sigma(R)$ and $s\in\Sigma(R')$ for any $k$-algebra $R$ and any $R$-algebra $R'$. Then $\f^\Sigma$ is a $G$-equivariant morphism of proalgebraic groups. Moreover, if $\f$ is a closed embedding, so is $\f^\Sigma$. It follows that $\f\colon A\to B$ induces a morphism
$$\f\wr G \colon A\wr G\to B\wr G$$
by $(\f\wr G)(f,g)=(\f^\Sigma(f),g)$ for $f\in (A^\Sigma)(R)$, $g\in G(R)$ and $R$ a $k$-algebra. In other words, the wreath product is functorial in the first component. Moreover, if $A\leq B$ and $G\leq H$ are closed subgroups, then $A\wr G$ is a closed subgroup of $B\wr H$. 


\subsection{The embedding theorem}
\label{subsec: The embedding theorem}

Let $H$ be a closed subgroup of a proalgebraic group $G$. Assume that $\Sigma=H\backslash G$ is a finite $k$-scheme and that the canonical map $\pi\colon G\to \Sigma$ has a section $T\colon \Sigma\to G$, i.e., $T$ is a morphism of $k$-schemes such that $\pi T=\id_\Sigma$. Note that if $G$ and $H$ were abstract groups, then such a $T$ would correspond to a right transversal of $H$ in $G$, i.e., a complete system of representatives of the right cosets of $H$ in $G$.
Let $\rho\colon G\to \Aut(\Sigma)^{\op}$ denote the morphism of proalgebraic groups corresponding to the natural right action of $G$ on $\Sigma$. Here $\Aut(\Sigma)^{\op}$ denotes the opposite group of $\Aut(\Sigma)$.
(Recall that by \cite[Chapter II, \S1, 2.7]{DemazureGabriel:GroupesAlgebriques} $\Aut(\Sigma)$ is an algebraic group). The main goal of this section is to construct a closed embedding of $G$ into $H\wr \rho(G)=H^\Sigma\rtimes \rho(G)$.

Let $d\colon G\to G^\Sigma$ denote the diagonal embedding (as in Lemma \ref{lemma: AhochSigma}). Since $G$ acts trivially on $d(G)\leq G^\Sigma$, we have a closed embedding $G\xrightarrow{d\times \rho}G\wr\rho(G)$ of proalgebraic groups. Note that $T\in (G^\Sigma)(k)\leq (G\wr \rho(G))(k)$ and so conjugation with $T$ defines an automorphism $i_T\colon G\wr \rho(G)\to G\wr\rho(G)$. Let 
$$\f\colon G\xrightarrow{d\times\rho} G\wr\rho(G)\xrightarrow{i_T} G\wr\rho(G)$$
be the composition of these two maps. So for a $k$-algebra $R$ and $g\in G(R)$ we have  $$\f(g)=(T,1) (d(g),\rho(g))(T^{-1},1)=(Td(g)\rho(g)(T^{-1}),\rho(g)).$$
Moreover, for an $R$-algebra $R'$ and $s\in \Sigma(R')$, we have $$(Td(g)\rho(g)(T^{-1}))(s)=T(s)gT(sg)^{-1}.$$ Since $\pi(T(s)g)=\pi(T(s))g=sg=\pi(T(sg))$, it follows from Lemma \ref{lemma: quotients} that $T(s)gT(sg)^{-1}\in H(R')$. Therefore, $\f(G)\leq H\wr \rho(G)\leq G\wr\rho(G)$. In summary, we have proved the following embedding theorem:
\begin{theo} \label{theo: embedding theorem}
	Let $H$ be a closed finite index subgroup of a proalgebraic group $G$. Let $T\colon \Sigma\to G$ be a section of the canonical map $G\to\Sigma=H\backslash G$. Then there exists a closed embedding of proalgebraic groups
	$\f\colon G\to H\wr \rho(G)$, given by $\f(g)=(f_g,\rho(g))$, where, for any $k$-algebra $R$ and $g\in G(R)$ the morphism $f_g\colon \Sigma_R\to H_R$, is given by $f_g(s)=T(s)gT(sg)^{-1}$ for $s\in \Sigma(R')$ and $R'$ an $R$-algebra.  
	
\end{theo}

As above, let $H$ be a closed finite index subgroup of a proalgebraic group $G$, $\Sigma=H\backslash G$ and $\rho\colon G\to \Aut(\Sigma)^{\op}$ the morphism of proalgebraic groups corresponding to the right action of $G$ on $\Sigma$. For any proalgebraic group $A$, we can consider the wreath product $A\wr \rho(H)=A^\Sigma\rtimes \rho(H)$. We define a morphism
$$\pi_A\colon A\wr \rho(H)\to A$$
by $\pi_A(f,g)=f(\overline{1})$ for $f\in (A^\Sigma)(R)=\Hom(\Sigma_R,A_R)$ and $g\in \rho(H)(R)$ where $R$ is a $k$-algebra and $\overline{1}\in \Sigma(R)$ is the trivial class, i.e., the image of $1\in G(R)$ under the canonical map $G(R)\to \Sigma(R)$. Note that for (abstract) groups $\G,\N, \N'$ and a morphism $\psi\colon \N\to \N'$ of groups, the map $\N\rtimes \G\to \N',\ (n,g)\mapsto \psi(n)$ is a morphism of groups, if and only if $\psi(g(n))=\psi(n)$ for all $g\in \G$ and $n\in \N$. For a $k$-algebra $R$, $h\in H(R)$ and $f\in \Hom(\Sigma_R,A_R)$ we have $h(f)(\overline{1})=f(\overline{1}h)=f(\overline{1})$. Therefore, $\pi_A$ is indeed a morphism of proalgebraic groups. In the following lemma, we consider the special case $A=H$.

\begin{lemma} \label{lemma: gives id}
Assume that in the setting of Theorem~\ref{theo: embedding theorem} the section $T\colon \Sigma\to G$ is such that $T(\overline{1})=1$. Then the restriction $\f'\colon H\to H\wr\rho(H)$ of the morphism $\f\colon G\to H\wr \rho(G)$ is such that the composition $H\xrightarrow{\f'}H\wr\rho(H)\xrightarrow{\pi_H}H$ is the identity.
\end{lemma}
\begin{proof}
	For a $k$-algebra $R$ and $h\in H(R)$ we have
	$$\pi_H(\f'(h))=f_h(\overline{1})=T(\overline{1})hT(\overline{1}h)^{-1}=T(\overline{1})hT(\overline{1})^{-1}=h.$$
\end{proof}

\subsection{Free proalgebraic groups}

In this section we recall the definition of free proalgebraic groups from \cite{Wibmer:FreeProalgebraicGroups} and discuss some of their properties needed in the following section.

Let $X$ be a set and $G$ a proalgebraic group. A map $\varphi\colon X\to G(\kb)$ (of sets) \emph{converges to $1$} if for every coalgebraic subgroup $N$ of $G$ the set $X\smallsetminus \varphi^{-1}(N(\kb))$ is finite. In other words, $\varphi$ converges to one if almost all elements of $X$ map to $1$ in any algebraic quotient of $G$.
A pair $(\iota,\Gamma)$ consisting of a proalgebraic group $\Gamma$ and a map $\iota\colon X\to \Gamma(\kb)$ converging to $1$ is a \emph{free proalgebraic group on $X$} if it satisfies the following universal property: If $G$ is a proalgebraic group and $\varphi\colon X\to G(\kb)$ a map converging to $1$, then there exists a unique morphism $\psi\colon\Gamma\to G$ of proalgebraic groups such that
$$
\xymatrix{
	X \ar^-{\iota}[rr] \ar_-{\varphi}[rd] & & \Gamma(\kb) \ar^-{\psi}[ld] \\
	& G(\kb) &
}
$$
commutes. According to \cite[Theorem 2.17]{Wibmer:FreeProalgebraicGroups} there exists a free proalgebraic group on $X$. Clearly, it is unique up to a unique isomorphism. We therefore take the liberty to sometimes speak of ``the'' free proalgebraic group on $X$, rather than of ``a'' free proalgebraic group on $X$. On the other hand, the phrase ``$\Gamma$ is a free proalgebraic group'' means that there exists a map $\iota$ such that $(\iota,\Gamma)$ is a free proalgebraic group on a suitable set $X$.

\begin{lemma} \label{lemma: embed free group}
	Let $(\iota,\Gamma)$ be the free proalgebraic group on a set $X$ and let $F_X$ be the (abstract) free group on $X$.  Then the map $F_X\to\Gamma(\kb)$ determined by $x\mapsto \iota(x)$ is injective.
\end{lemma}
\begin{proof}
	Let us first assume that $X$ is finite. In characteristic zero, one could use the well-known fact that $\Sl_2(\mathbb{Q})$ contains a free subgroup on two elements and therefore also a free subgroup on any finite number of elements. 
	We will give a proof, in any characteristic, based on the fact that the free group $F_X$ is residually finite (\cite[Prop. 17.5.11]{FriedJarden:FieldArithmetic}). This means that the intersection of all normal finite index subgroups of $F_X$ is trivial. For every normal finite index subgroup $N$ of $F_X$ let $G_N$ be the constant (\'{e}tale) algebraic group corresponding to the finite group $F_X/N$ and set $\varphi_N\colon X\to F_X\to F_X/N=G_N(k)=G_N(\kb)$. Because $X$ is finite, $\varphi_N$ converges to $1$ and there exists a morphism $\psi_N\colon \Gamma\to G_N$ of proalgebraic groups such that
	$$
	\xymatrix{
	 X \ar^-\iota[rr] \ar_-{\varphi_N}[rd] & & \Gamma(\kb) \ar^-{\psi_N}[ld] \\
	 & G_N(\kb) &
	}
	$$
	commutes. This yields a commutative diagram
		$$
	\xymatrix{
		F_X \ar[rr] \ar[rd] & & \Gamma(\kb) \ar^-{\psi_N}[ld] \\
		& G_N(\kb) &
	}
	$$
	that shows that an element in the kernel of $F_X\to \Gamma(\kb)$ lies in $N$. Since the intersection of all $N$'s is trivial, the kernel of $F_X\to \Gamma(\kb)$ is trivial. This proves the lemma for finite $X$.
	
	Let us now assume that $X$ is arbitrary. Assume that $f\in F_X$ lies in the kernel of $F_X\to \Gamma(\kb)$. Then there exists a finite subset $X'$ of $X$ such that $f\in F_{X'}\subseteq F_X$. Let $(\iota',\Gamma')$ be the free proalgebraic group on $X'$ and define $\varphi\colon X\to \Gamma'(\kb)$ by
	$$\varphi(x)=\begin{cases}
	\iota'(x) & \text{ if } x\in X' ,\\
	1 & \text{ if } x\notin X'.
	\end{cases}
$$	
Then $\varphi$ converges to $1$ and there exist a unique morphism $\psi\colon \Gamma\to \Gamma'$ of proalgebraic groups such that 	
	$$
\xymatrix{
	X \ar^-{\iota}[rr] \ar_-{\varphi}[rd] & & \Gamma(\kb) \ar^-{\psi}[ld] \\
	& \Gamma'(\kb) &
}
$$
commutes. This yields a commutative diagram 
$$
\xymatrix{
F_X \ar[r] \ar[d] & \Gamma(\kb) \ar^-{\psi}[d]\\
F_{X'} \ar[r] & \Gamma'(\kb)	
}
$$
where $F_X\to F_{X'}$ restricts to the identity on $F_{X'}\subseteq F_X$. As $F_{X'}\to \Gamma'(\kb)$ is injective, we see that necessarily $f=1$. Thus $F_X\to\Gamma(\kb)$ is injective.	
%
%
\end{proof}

The above lemma shows, in particular, that the map $\iota\colon X\to \Gamma(\kb)$ is injective. We will therefore from now on identify $X$ with a subset of $\Gamma(\kb)$.

%


\begin{lemma} \label{lemma: surjective to quotient}
	Assume that $k$ is algebraically closed.  Let $H$ be a closed finite index subgroup of the free proalgebraic group $\Gamma$ on a set $X$. Then the map $F_X\to (H\backslash\Gamma)(k)$ from the (abstract) free group on $X$ to $(H\backslash\Gamma)(k)$ is surjective.
\end{lemma}
\begin{proof}
Let $\langle X\rangle$ denote the smallest closed subgroup of $\Gamma$ such that $X\subseteq \langle X\rangle(k)$. This is well-defined because the intersection of closed subgroups containing $X$ is a closed subgroup containing $X$. According to \cite[Theorem 2.17]{Wibmer:FreeProalgebraicGroups} we have $\langle X\rangle =\Gamma$. By \cite[Remark 2.20]{Wibmer:FreeProalgebraicGroups} the proalgebraic group $\Gamma$ is reduced. Therefore $H\backslash\Gamma$ is a reduced finite $k$-scheme by Lemma \ref{lemma: quotients} (iii). So, geometrically, $H\backslash\Gamma$ is simply a finite set of $k$-points with no additional structure. Let $s_1,\ldots,s_n$ denote these $k$-points and let $\pi\colon \Gamma\to H\backslash\Gamma$ be the canonical map. By Lemma \ref{lemma: quotients} (ii) the map $\pi\colon \Gamma(k)\to (H\backslash\Gamma)(k)$ is surjective. So $\Gamma$ is the disjoint union of the non-empty open and closed subschemes $\Gamma_i=\pi^{-1}(s_i)$ $(i=1,\ldots,n)$. Suppose there exists an $i\in \{1,\ldots,n\}$ such that $F_X\cap \Gamma_i(k)=\emptyset$. Then $$F_X\subseteq \bigcup_{j=1 \atop j\neq i}^n\Gamma_i(k) \text{ and } \bigcup_{j=1 \atop j\neq i}^n\Gamma_i$$
is a proper closed subscheme of $\Gamma$ containing $F_X$ in its $k$-points. As in \cite[Theorem 4.3, (a)]{Waterhouse:IntroductiontoAffineGroupSchemes} one sees that $\langle X\rangle$ is the smallest closed subscheme of $\Gamma$ containing $F_X$ in its $k$-points. We thus obtain a contradiction to $\langle X\rangle =\Gamma$. Thus $F_X\cap\Gamma_i(k)\neq \emptyset$ for every $i$ and so $F_X\to (H\backslash\Gamma)(k)$ is surjective. 
\end{proof}

\subsection{Finite index subgroups of free proalgebraic groups are free} \label{subsec: Finite index subgroups of free proalgebraic groups are free}

In this section we show that a closed finite index subgroup of a free proalgebraic group is itself free.
For the proof we need to recall the notion of \emph{Schreier transversal} (see, e.g., \cite[Section~ 3.1]{RibesSteinberg:AWreathProductApproachToClassicalSubgroupTheorems}). Let $F_X$ be the (abstract) free group on a set $X$ and let $H$ be a subgroup of $F_X$. A \emph{Schreier transversal} of $H$ in $F_X$ is a subset $T$ of $F_X$ such that
\begin{itemize}
	\item $T$ is a complete system of representatives of the right cosets of $H$ in $F_X$, i.e., the map $T\to H\backslash F_X,\ t\mapsto Ht$ is bijective and
	\item $T$ is closed under taking prefixes, i.e., if $x_1,\ldots,x_n\in X\cup X^{-1}$ such that $x_1\ldots x_n$ is in reduced form and belongs to $T$, then also $x_1\ldots x_i$ belongs to $T$ for $i=0,\ldots,n-1$. (In particular, $1\in T$.)
\end{itemize}
Schreier transversals exist for every subgroup $H$ of $F_X$ (\cite[Lemma 3.1.1]{RibesSteinberg:AWreathProductApproachToClassicalSubgroupTheorems}). 
Let $\mu\colon F_X\to T$ be the unique map such that $H\omega=H\mu(\omega)$ for all $\omega\in F_X$.
The following theorem explains the main use of Schreier transversals. See \cite[Theorem 3.6.1]{RibesZalesskii:ProfiniteGroups} or \cite[Section 17.5]{FriedJarden:FieldArithmetic}.
\begin{theo}[Nielsen-Schreier] \label{theo: NielsenSchreier}
	Let $H$ be a subgroup of the free group $F_X$ and let $T$ be a Schreier transversal of $H$ in $F_X$. Set
	$$B=\{tx\mu(tx)^{-1}|\ t\in T, \ x\in X,\ tx\mu(tx)^{-1}\neq 1 \}.$$
	Then $H$ is a free group on $B$. Furthermore, if $[F_X:H]$ is finite, then $|B|=1+[F_X:H](|X|-1)$. 
\end{theo}

We are now prepared to prove our proalgebaic analog of the Nielsen-Schreier Theorem.

\begin{theo} \label{theo: proalgebraic NielsenSchreier}
	Assume that $k$ is algebraically closed. Let $H$ be a closed finite index subgroup of the free proalgebraic group $\Gamma$ on the set $X$. Then $H$ is a free proalgebraic group. Moreover, if $X$ is finite, then $H$ is a free proalgebraic group on a set of cardinality $1+[\Gamma:H](|X|-1)$. If $|X|$ is infinite, then $H$ is a free proalgebraic group on a set of cardinality $|X|$.
\end{theo}
\begin{proof}
	Let $F_X$ denote the subgroup of $\Gamma(k)$ generated by $X$. We know from Lemma \ref{lemma: embed free group} that $F_X$ is the (abstract) free group on $X$. Let $\overline{T}$ be a Schreier transversal of $F_X\cap H(k)$ in $F_X$. 
	Let $\mu\colon F_X\to \overline{T}$ be the unique map such that $(F_X\cap H(k))\omega=(F_X\cap H(k))\mu(\omega)$ for any $\omega\in F_X$. Set
	$$B=\left\{tx\mu(tx)^{-1}|\ t\in \overline{T},\ x\in X,\ tx\mu(tx)^{-1}\neq 1 \right\}.$$
	Note that $tx\mu(tx)^{-1}\in F_X\cap H(k)$ because $(F_X\cap H(k))tx=(F_X\cap H(k))\mu(tx)$. In particular, $B\subseteq H(k)$. We will prove that $H$ is a free proalgebraic group by showing that the inclusion map $\iota\colon B\to H(k)$ satisfies the universal property. 
	
Let us first show that $\iota$ converges to $1$.
Let $N$ be a coalgebraic subgroup of $H$. By \cite[Lemma 2.7]{Wibmer:FreeProalgebraicGroups} there exists a coalgebraic subgroup $N'$ of $\Gamma$ such that $H\cap N'\subseteq N$. Then $x\in N'(k)$  for a subset $X'$ of $X$ with $X\smallsetminus X'$ finite. Since $N'$ is normal in $\Gamma$, we see that $txt^{-1}\in N'(k)$ for $x\in X'$ and $t\in \overline{T}$. 

By Lemma \ref{lemma: exists coalgebraic subgroup is finite index subgroup} there exists a coalgebraic subgroup $N''$ of $\Gamma$ such that $N''\subseteq H$. In particular, $x\in N''(k)$ for all $x\in X''$, where $X''$ is a subset of $X$ with $X\smallsetminus X''$ finite. Because $N''$ is normal in $\Gamma$, we have $txt^{-1}\in N''(k)\leq H(k)$ for all $x\in X''$ and $t\in\overline{T}$. So $txt^{-1}\in F_X\cap H(k)$, i.e., $\mu(tx)=\mu(t)=t$ for all $x\in X''$ and $t\in \overline{T}$.


For $x\in X'\cap X''$ and $t\in\overline{T}$ we have $tx\mu(tx)^{-1}=txt^{-1}\in N'(k)$ and so $tx\mu(tx)^{-1}\in H(k)\cap N'(k)\leq  N(k)$. Thus $\iota$ converges to $1$.

To verify the universal property, let $G$ be a proalgebraic group and let $\varphi\colon B\to G(k)$ be a map converging to $1$.
As in Section \ref{subsec: The embedding theorem}, set $\Sigma=H\backslash\Gamma$ and let the morphism $\rho\colon \Gamma\to \Aut(\Sigma)^{\op}$ of proalgebraic groups be defined through the right action of $\Gamma$ on $\Sigma$. 
By Lemma \ref{lemma: surjective to quotient} the map $F_X\to \Sigma(k)$ is surjective and therefore induces bijections 
\begin{equation} \label{eq:bijection}
\overline{T}\simeq (F_X\cap H(k))\backslash F_X\to \Sigma(k).
\end{equation}
As in the proof of Lemma \ref{lemma: surjective to quotient}, the $k$-scheme $\Sigma$ is geometrically simply a finite set of $k$ points with no additional structure. So, to specify a morphism from $\Sigma$ into a $k$-scheme $X$ is equivalent to specifying some $k$-points of $X$, one for each $k$-point of $\Sigma$. Thus there exists a morphism $T\colon \Sigma\to \Gamma$ such that 
$T\colon \Sigma(k)\to \Gamma(k)$ is the inverse of the bijection (\ref{eq:bijection}) followed by the inclusion $\overline{T}\to \Gamma(k)$. In particular, $T\colon\Sigma\to \Gamma$ is a section of the canonical map $\pi\colon \Gamma\to H\backslash\Gamma$ with $T(\overline{1})=1$ and we have a commutative diagram
\begin{equation} \label{eq: com dia}
\xymatrix{
F_X \ar^-\mu[r] \ar_-\pi[d] & \overline{T} \\
\Sigma(k) \ar_-T[ur]	
}
\end{equation}
We consider the closed embedding $\f\colon\Gamma\to H\wr\rho(\Gamma)$ as in Theorem \ref{theo: embedding theorem}.

We have to show that there exists a unique morphism $\psi\colon H\to G$ of proalgebraic groups such that $\psi(b)=\varphi(b)$ for all $b\in B$. We first establish the uniqueness of $\psi$. 
Using Lemma \ref{lemma: gives id} and the basic properties of wreath products, we obtain the commutative diagram
$$
\xymatrix{
\Gamma \ar^-\f[r] & H\wr\rho(\Gamma) \ar^-{\psi\wr\rho(\Gamma)}[r] & G\wr \rho(\Gamma) \\
H \ar@{^(->}[u] \ar^-{\f|_H}[r] \ar_-\id[rd] & H\wr \rho(H) \ar^{\pi_H}[d] \ar@{^(->}[u] \ar^-{\psi\wr \rho(H)}[r] & G\wr \rho(H) \ar@{^(->}[u] \ar^-{\pi_G}[d] \\
& H \ar^-\psi[r] & G	
}
$$
It follows that $\psi=\pi_G\circ (\psi\wr\rho(H))\circ\f|_H.$ For $x\in X\subseteq\Gamma(k)$ we have
\begin{equation} \label{eq: values on X}
(\psi\wr\rho(\Gamma))(\f(x))=(\psi f_x,\rho(x)),
\end{equation} where, as in Theorem \ref{theo: embedding theorem}, $f_x\colon \Sigma\to H$ is given by $f_x(s)=T(s)xT(sx)^{-1}$ for $s\in\Sigma(R)$ and $R$ a $k$-algebra. Let $s_1,\ldots,s_n$ be the $k$-points of $\Sigma$ and set $t_i=T(s_i)$ for $i=1,\ldots,n$. Then $\overline{T}=\{t_1,\ldots,t_n\}$ and $\pi(t_i)=s_i$. The morphism $\psi f_x\colon \Sigma\to G$ is completely determined by its values on $s_1,\ldots,s_n$. Moreover, 
$$T(s_ix)=T(\pi(t_i)x)=T(\pi(t_i x))=\mu(t_ix)$$
by (\ref{eq: com dia}) and so 
\begin{equation} \label{eq: formula}
(\psi f_x)(s_i)=\psi(t_ix\mu(t_ix)^{-1})=\varphi(t_ix\mu(t_ix)^{-1}),
\end{equation}
where we have extended $\varphi\colon B\to G(k)$ to $B\cup \{1\}\to G(k)$ by $\varphi(1)=1$.
Like any morphism on $\Gamma$, the morphism $(\psi\wr\rho(\Gamma))\circ\f$ is completely determined by its values on $X$. But then (\ref{eq: values on X}) and (\ref{eq: formula}) show that $(\psi\wr\rho(\Gamma))\circ\f$ is completely determined by $\varphi$. Thus  $\psi=\pi_G\circ (\psi\wr\rho(H))\circ\f|_H$ is completely determined by $\varphi$. This proves the uniqueness of $\psi$.

Let us next establish the existence of $\psi$. For $x\in X$, let $g_x\colon \Sigma\to G$ be the morphism of $k$-schemes determined by $g_x(s_i)=\varphi(t_ix\mu(t_ix)^{-1})=\varphi(f_x(s_i))$ for $i=1,\ldots,n$
 and consider the map $\chi\colon X\to (G\wr\rho(\Gamma))(k)$ defined by $\chi(x)=(g_x,\rho(x))$. To show that $\chi$ converges to $1$ we need an observation about coalgebraic subgroups: Assume that $N$ is coalgebraic subgroup of $G^n$ and for $i=1,\ldots,n$ let $\lambda_i\colon G\to G^n$ denote the inclusion into the $i$-th factor. Then $\lambda_i^{-1}(N)$ is a coalgebraic subgroup of $G$ and therefore also $N_1=\lambda_1^{-1}(N)\cap\ldots\cap\lambda_n^{-1}(N)$ is a coalgebraic subgroup of $G$. Moreover, $N_1^n\leq N$.

Now let $N$ be a coalgebraic subgroup of $G\wr\rho(\Gamma)$. We have to show that $\chi$ maps all but finitely many $x\in X$ into $N(k)$. The intersection $N_1=N\cap G^\Sigma$ of $N$ with the subgroup $G^\Sigma$ of $G\wr\rho(\Gamma)=G^\Sigma\rtimes\rho(\Gamma)$ is a coalgebraic subgroup of $G^\Sigma$. Since $G^\Sigma$ can be identified with $G^n$, we can use the above observation to find a coalgebraic subgroup $N_2$ of $G$ such that $N_2^\Sigma\leq N_1$. Because $\varphi\colon B\to G(k)$ converges to $1$, we have $g_x(s_i)=\varphi(t_ix\mu(t_ix)^{-1})\in N_2(k)$ for all but finitely many pairs $(i,x)\in \{1,\ldots,n\}\times X$. Then $g_x\in N_2^\Sigma(k)\leq N_1(k)\leq N(k)$ for all but finitely many $x$. Because $\rho(\Gamma)\leq \Aut(\Sigma)^{\op}$ is an algebraic group, $\rho(x)=1$ for all but finitely many $x\in X$. Therefore $\chi(x)=(g_x,\rho(x))\in N(k)$ for all but finitely many $x\in X$. So $\chi$ converges to $1$ and from the universal property of $\Gamma$ we obtain a morphism $\tau\colon \Gamma \to G\wr\rho(\Gamma)$ of proalgebraic groups such that $\tau(x)=\chi(x)$ for all $x\in X$. Note that the composition $\Gamma\xrightarrow{\tau}G\wr\rho(\Gamma)\to \rho(\Gamma)$ is simply $\rho$ because this in true on $X$. It follows that $\tau$ maps $H$ into $G\wr\rho(H)$. We will show that $\psi=\pi_G\circ\tau|_H$ extends $\varphi$. That is, we have to show that $\psi(b)=\varphi(b)$ for all $b\in B$. This is really just a computation. In fact, it is the same computation as in the proof of \cite[Theorem 3.2.1]{RibesSteinberg:AWreathProductApproachToClassicalSubgroupTheorems}. For the sake of completeness we include the details.

For $\omega\in F_X\subseteq \Gamma(k)$ let $g_\omega\colon \Sigma\to G$ be defined by $\tau(\omega)=(g_\omega,\rho(\omega))\in (G\wr\rho(\Gamma))(k)$. Note that this is compatible with the above definition of $g_x$ for $x\in X$. For $\omega_1,\omega_2\in F_X$ we have
$$(g_{\omega_1\omega_2},\rho(\omega_1\omega_2))=\tau(\omega_1\omega_2)=\tau(\omega_1)\tau(\omega_2)=(g_{\omega_1},\rho(\omega_1))(g_{\omega_2},\rho(\omega_2))=(g_{\omega_1}\rho(\omega_1)(g_{\omega_2}),\rho(\omega_1)\rho(\omega_2)).$$
Therefore $g_{\omega_1\omega_2}=g_{\omega_1}\cdot\rho(\omega_1)(g_{\omega_2})$ and inductively we find that
\begin{equation} \label{eq: expansion formula}
g_{\omega_1\ldots\omega_m}=g_{\omega_1}\cdot \rho(\omega_1)(g_{\omega_2})\cdot \rho(\omega_1\omega_2)(g_{\omega_3})\cdot \ldots\cdot \rho(\omega_1\ldots\omega_{m-1})(g_{\omega_m})
\end{equation}
for $\omega_1,\ldots,\omega_m\in F_X$. Similarly, for $\omega\in F_X$ we have
$$(1,1)=\tau(1)=\tau(\omega^{-1}\omega)=\tau(\omega^{-1})\tau(\omega)=(g_{\omega^{-1}},\rho(\omega^{-1}))(g_{\omega},\rho(\omega))=(g_{\omega^{-1}}\rho(\omega^{-1})(g_\omega),1)$$ and therefore
\begin{equation} \label{eq: g for inverse}
g_{\omega^{-1}}=\rho(\omega^{-1})(g_\omega)^{-1}.
\end{equation}

Let $b\in B$. We have to show that $\psi(b)=\varphi(b)$. Write $b=tx\mu(tx)^{-1}$ for some $x\in X$ and $t\in\overline{T}\subseteq F_X$.
Write $t=x_1\ldots x_{\ell-1}\in F_X$ and $\mu(tx)^{-1}=x_{\ell+1}\ldots x_m\in F_X$ in reduced form and set $x_\ell=x$. Then $b=x_1\ldots x_m$. Furthermore, set $u_i=\mu(x_1\ldots x_i)$ for $i=1,\ldots, m$ and $u_0=1$. Since $\overline{T}$ is prefixed closed, we have $u_i=x_1\ldots x_i$ for $i<\ell$. On the other hand, for $i>\ell$, using  $\mu(tx)^{-1}=x_{\ell+1}\ldots x_m$, we find
$$ \pi(u_i)=\pi(txx_{\ell+1}\ldots x_i)=\pi(x_m^{-1}\ldots x_{\ell+1}^{-1}x_{\ell+1}\ldots x_i)=\pi(x_m^{-1}\ldots x_{i+1}^{-1}).$$
Therefore, using again that $\overline{T}$ is prefixed closed, we obtain $u_i=x_m^{-1}\ldots x_{i+1}^{-1}$ for $i>\ell$.
From these explicit formulas for the $u_i$'s we derive that 
\begin{equation} \label{eq: ui}
u_{i-1}x_iu_i^{-1}=1 \quad \text{ for }\quad  i\in\{1,\ldots,m\}\smallsetminus\{\ell\}.
\end{equation}  
We next show that for $y\in X\cup X^{-1}$ and $\omega\in F_X$
\begin{equation} \label{eq: is 1}
g_y(\pi(\omega))=1 \quad \text{ if } \quad \mu(\omega)y\mu(\omega y)^{-1}=1.
\end{equation}
Recall that for $y\in X$ we have $g_y(\pi(\omega))=\varphi(T(\pi(\omega))y\mu(T(\pi(\omega))y)^{-1})=\varphi(\mu(\omega)y\mu(\mu(\omega) y)^{-1})$.
But $\pi(\mu(\omega))=\pi(\omega)$ and so $\pi(\mu(\omega)y)=\pi(\omega y)$. Therefore $\mu(\mu(\omega)y)=\mu(\omega y)$. This proves (\ref{eq: ui}) for $y\in X$. If $y\in X^{-1}$ then $y^{-1}\in X$ and using (\ref{eq: g for inverse}) we obtain
\begin{align*}
g_y(\pi(\omega)) & =g_{y^{-1}}(\pi(\omega)y)^{-1}=g_{y^{-1}}(\pi(\omega y))^{-1}=\varphi(\mu(\omega y)y^{-1}\mu(\omega y y^{-1})^{-1})^{-1}=\\
& =\varphi(\mu(\omega y)y^{-1}\mu(\omega)^{-1})^{-1}=\varphi(1)^{-1}=1.
\end{align*}
This proves (\ref{eq: is 1}).

Since, $\mu(\omega)x_i\mu(\omega x_i)^{-1}=u_{i-1}x_iu_i^{-1}=1$ by (\ref{eq: ui}) for $i\in\{1,\ldots,m\}\smallsetminus\{\ell\}$ and $\omega=x_1\ldots x_{i-1}$, we obtain $g_{x_i}(\pi(x_1\ldots x_{i-1}))=1$ for $i\neq \ell$ from (\ref{eq: is 1}). Using this and (\ref{eq: expansion formula}) we find
\begin{align*}
\psi(b)&=\pi_G(\tau(b))=g_b(\pi(1))=g_{x_1\ldots x_m}(\pi(1))=\\
&=g_{x_1}(\pi(1))\cdot \rho(x_1)(g_{x_2})(\pi(1))\cdot\ldots\cdot \rho(x_1\ldots x_{m-1})(g_{x_m})(\pi(1))=\\
&=g_{x_1}(\pi(1))\cdot g_{x_2}(\pi(x_1))\cdot \ldots\cdot g_{x_m}(\pi(x_1\ldots x_{m-1}))=\\
&=g_{x_\ell}(\pi(x_1\ldots x_{\ell-1}))=\varphi(\mu(x_1\ldots x_{\ell-1})x_\ell\mu(x_1\ldots x_\ell)^{-1})=\\
&=\varphi(x_1\ldots x_m)=\varphi(b).
\end{align*}
So $\psi$ extends $\varphi$ as claimed and we have shown that $H$ is the free proalgebraic group on $B$. It remains to determine the cardinality of $B$. However, since $\overline{T}$ is a Schreier transversal of $F_X\cap H(k)$ in $F_X$ and $[F_X:(F_X\cap H(k))]=[\Gamma:H]$ by (\ref{eq:bijection}), this follows from Theorem \ref{theo: NielsenSchreier}.
\end{proof}

\section{Matzat's conjecture}

In this section we prove Matzat's conjecture for fields of constants of infinite transcendence degree. 
Throughout this section $(K,\de)$ denotes an (ordinary) differential field of characteristic zero. We assume that the field of constants $k=K^\de=\{a\in K|\ \de(a)=0\}$ is algebraically closed.

\subsection{Differential Galois theory}

We begin by recalling the basic definitions and results from differential Galois theory. The reader interested in more background is invited to consult \cite{BachmayrHarbaterHartmannWibmer:FreeDifferentialGaloisGroups} or any of the introductory textbooks \cite{Magid:LecturesOnDifferentialGaloisTheory}, \cite{SingerPut:differential}, \cite{CrespoHajto:AlgebraicGroupsAndDifferentialGaloisTheory} \cite{Sauloy:DifferentialGaloisTheoryThroughRiemannHilbertCorrespondence}.
We consider a family 
\begin{equation} \label{eq: equations}
(\de(y)=A_iy)_{i\in I}, \quad A_i\in K^{n_i\times n_i}
\end{equation}
of linear differential systems over $K$ indexed by a set $I$.

\begin{defi}
	A differential field extension $L/K$ is a \emph{Picard-Vessiot extension} for (\ref{eq: equations}) if $L^\de=k$ and for every $i\in I$ there exists $Y_i\in \Gl_n(L)$ with $\de(Y_i)=A_iY_i$ and $L$ is generated as a field extension of $K$ by the entries of all $Y_i$'s. The $K$-subalgebra $R$ of $L$ generated by the inverse of the determinant and all the entries of all $Y_i's$ is called a \emph{Picard-Vessiot ring} for (\ref{eq: equations}).
\end{defi}
There exists a Picard-Vessiot extension for (\ref{eq: equations}) and it is unique up to a $K$-$\de$-isomorphism (since $k$ is algebraically closed). 
We will call the Picard-Vessiot extension $\widetilde{K}/K$ for the family of all linear differential systems over $K$ the \emph{linear closure} of $K$. In the literature, the term Picard-Vessiot extension is often only used for a single equation. To make this distinction, we say that a Picard-Vessiot extension is of finite type if it is the Picard-Vessiot extension for a single linear differential system. Note that every Picard-Vessiot extension is a directed union Picard-Vessiot extensions of finite type.

Recall that an element $a$ in a $\de$-field extension $L$ of $K$ is called \emph{$\de$-finite} over $K$ if $a$ satisfies a non-trivial homogeneous linear differential equation with coefficients in $K$. Note that $a$ is $\de$-finite over $K$ if and only if there exists a finite dimensional $K$-subspace $V$ of $L$ such that $a\in V$ and $\de(V)\subseteq V$.
The following characterization of the linear closure will be useful later on.

\begin{lemma} \label{lemma:characterize linear closure}
	Let $L/K$ be an extension of $\de$-fields such that $L^\de=k$. Then $L$ is a linear closure of $K$ if and only if
	\begin{enumerate}
		\item for every $n\geq 1$ and $A\in K^{n\times n}$ there exists $Y\in\Gl_n(L)$ such that $\de(Y)=AY$ and
		\item $L$ is generated as a field extension of $K$ by elements that are $\de$-finite over $K$.
	\end{enumerate}
\end{lemma}
\begin{proof}
	If $Y\in\Gl_n(L)$ such that $\de(Y)=AY$ for some $A\in K^{n\times n}$, then the $K$-subspace $V$ of $L$ generated by the entries of $Y$ satisfies $\de(V)\subseteq V$. Thus the linear closure of $K$ satisfies the above two conditions. 
	
	Conversely, assume that $L/K$ satisfies (i) and (ii). Let $L'$ be the subfield of $L$ generated over $K$ by all entries of all matrices $Y\in\Gl_n(L)$ such that $\de(Y)=AY$ for some $A\in\Gl_n(L)$ $(n\geq 1)$. We have to show that $L=L'$. Since $L$ is generated by elements that are $\de$-finite over $K$, it suffices to show that  every $y\in L$ that is $\de$-finite over $K$ is contained in $L'$. But if $y$ satisfies a linear differential equation
	\begin{equation} \label{eq: linear eq}
		\de^n(y)+a_{n-1}\de^{n-1}(y)+\ldots+a_0y=0
	\end{equation}
	over $K$, then $y$ lies in the subfield of $L$ generated by the entries of $Y$, where $Y\in\Gl_n(L)$ satisfies $\de(Y)=AY$, with $A$ the companion matrix of (\ref{eq: linear eq}). Therefore, $y\in L'$ and $L=L'$. 
\end{proof}

\begin{defi}
	Let $L/K$ be a Picard-Vessiot extension with Picard-Vessiot ring $R$. The \emph{differential Galois group} $G(L/K)$ of $L/K$ is the functor from the category of $k$-algebras to the category of groups, that associates to every $k$-algebra $T$ the group of all $K\otimes_k T$-$\de$-automorphisms of $R\otimes_k T$, where $T$ is considered as a constant $\de$-ring.
\end{defi}

The differential Galois group of a Picard-Vessiot extension is representable, i.e., it is a proalgebraic group.
The differential Galois group of a Picard-Vessiot extension of finite type is an algebraic group. The \emph{absolute differential Galois group} of $K$ is the differential Galois group of the linear closure $\widetilde{K}/K$.

The Galois correspondence for a Picard-Vessiot extension $L/K$ is the bijection $M\mapsto G(L/M)$ between the intermediate differential fields of $L/K$ and the closed subgroups of $G(L/K)$.  Moreover, an intermediate differential field $M$ is itself a Picard-Vessiot extension of $K$, if and only if $G(L/M)$ is normal in $G$. In this case the induced morphism $G(L/K)\to G(M/K)$ is a quotient map with kernel $G(L/M)$. We will need to know the following.

\begin{lemma} \label{lemma: finite PV implies quotient is finite scheme}
	Let $L/K$ be a Picard-Vessiot extension with differential Galois group $G$. Let $M$ be an intermediate differential field of $L/K$ such that $M/K$ is finite and let $H=G(L/M)\leq G$. Then $G/H$ is a finite $k$-scheme. 
\end{lemma}
\begin{proof}
	Let $M'/K$ be a Picard-Vessiot extension of finite type such that $M\subseteq M'\subseteq L$. Such an $M'$ exists because $M/K$ is finite and $L$ is the directed union of Picard-Vessiot extensions of finite type. Set $N=G(L/M)$. Since $G/H\simeq (G/N)/(H/N)$ and $G/N$ is the differential Galois group of $M'/K$, while $H/N$ is the differential Galois group of $M'/M$, we can reduce to the case that $L/K$ is of finite type. Then $G$ is an algebraic group and therefore $G/H$ is a scheme (\cite[Chapter III, \S 3, Theorem 5.4]{DemazureGabriel:GroupesAlgebriques}).
	Let $M^o$ denote the relative algebraic closure of $K$ in $L$. Then $G(L/M^o)=G^o$, the identity component of $G$ (\cite[Prop. 1.34]{SingerPut:differential}). As $M'\subseteq M^o$, we have $G^o=G(L/M^o)\leq G(L/M')=H$. Therefore $G/H$ is finite. 
\end{proof}

\subsection{Linear closures and the proof of Matzat's conjecture}

To prove Matzat's conjecture for function fields, we need more information about linear closures.

\begin{lemma} \label{lem: isomorphic absolute differential Galois groups}
	Let $(K,\de)$ be a differential field and let $a$ be a non-zero element of $K$. Then the differential fields $(K,\de)$ and $(K,a\de)$ have isomorphic absolute differential Galois groups.  
\end{lemma}
\begin{proof}
	Let $(\widetilde{K},\delta)$ denote the linear closure of $(K,\de)$. We will first show that $(\widetilde{K},a\de)$ is a linear closure of $(K,a\de)$. To begin, we observe that $\widetilde{K}^{a\de}=\widetilde{K}^{\de}=K^\de=K^{a\de}(=k)$. 
	To see that every linear differential system over $(K, a\de)$ has a fundamental solution matrix in $(\widetilde{K},a\de)$, fix $A\in K^{n\times n}$ and let $Y\in\Gl_n(\widetilde{K})$ be such that $\de(Y)=a^{-1}AY$. Then $Y$ is a fundamental solution matrix for the linear differential system over $(K,a\de)$ defined by $A$. As $\widetilde{K}$ is generated as a field extension of $K$ by the entries of fundamental solution matrices over $(K,\de)$ it is clear that $\widetilde{K}$ is also generated as a field extension of $K$ by the entries of fundamental solution matrices over $(K,a\de)$. Therefore $(\widetilde{K}, a \de)$ is a linear closure of $(K,a\de)$. Moreover, if $(R, \de)$ is the Picard-Vessiot ring of $(\widetilde{K}/K, \de)$, then $(R,a\de)$ is the Picard-Vessiot ring of $(\widetilde{K}/K, a\de)$.
	
	For a $k$-algebra $T$, a $K\otimes_k T$-automorphism of $R\otimes_k T$ commutes with $\de$ if and only if it commutes with $a\de$. Thus the two absolute differential Galois groups are isomorphic.  
\end{proof}

%
%
%
%
%
%
%
%

\begin{lemma} \label{lemma: finite de extension}
	Let $L/K$ be an extension of $\de$-fields such that $L/K$ is a finite field extension. Let $\widetilde{L}/L$ be a linear closure of $L$. Then $\widetilde{L}$ is also a linear closure of $K$.
\end{lemma}
\begin{proof}
	Clearly, for every $A\in K^{n\times n}\subseteq L^{n\times n}$ there exists $Y\in\Gl_n(\widetilde{L})$ with $\de(Y)=AY$. By Lemma~\ref{lemma:characterize linear closure} it suffices to show that $\widetilde{L}/K$ is generated as a field extension of $K$ by elements that are $\de$-finite over $K$.  
	
	Let $f\in \widetilde{L}$ be $\de$-finite over $L$. We will show that $f$ is also $\de$-finite over $K$. Let $V$ be a finite $L$-subspace of $\widetilde{L}$ such that $\de(V)\subseteq V$ and $f\in V$. Because $L/K$ is finite, $V$ is also finite as a $K$-vector space. Thus, $f$ is $\de$-finite over $K$.
	
	The field extension $\widetilde{L}/K$ is generated by the elements of $L$ and by elements of $\widetilde{L}$ that are $\de$-finite over $L$. As all these elements are $\de$-finite over $K$, we see that $\widetilde{L}/K$ is generated by elements that are $\de$-finite over $K$.  	
\end{proof}

We note that if $L/K$ and $M/L$ are Picard-Vessiot extensions, it is in general not possible to embed $M/K$ into a Picard-Vessiot extension of $K$. (See page 13 in \cite{Magid:ThePicardVessiotAntiderivativeClosure} for such an example.) However, it follows from Lemma \ref{lemma: finite de extension}, that it is possible to embed $M/K$ into a Picard-Vessiot extension of $K$ if $L/K$ is finite.  

\begin{cor} \label{cor: absolute Galois group is finite index subgroup}
	Let $L/K$ be an extension of $\de$-fields such that $L/K$ is a finite field extension. Then the absolute differential Galois group of $L$ is isomorphic to a closed finite index subgroup of the absolute differential Galois group of $K$. 
\end{cor}
\begin{proof}
	Let $\widetilde{L}$ be the linear closure of $L$. We know from Lemma \ref{lemma: finite de extension} that $\widetilde{L}$ is the linear closure of $K$. Thus the claim follows from Lemma \ref{lemma: finite PV implies quotient is finite scheme}.
\end{proof}	
We are now prepared to prove our main results.
\begin{prop} \label{prop: Matzat implies Matzat}
	Let $k$ be an algebraically closed field of characteristic zero. If Matzat's conjecture is true for $k(x)$, the rational function field in one variable equipped with the standard derivation $\frac{d}{dx}$, then Matzat's conjecture is true for any one-variable function field over $k$, equipped with a non-trivial $k$-derivation.
\end{prop}
\begin{proof}
	Let $K$ be a one-variable function field over $K$ equipped with a non-trivial $k$-derivation $\de$. Let $x\in K$ be transcendental over $k$. The standard derivation $\frac{d}{dx}$ on $k(x)$ extends uniquely to a derivation $\de_x\colon K\to K$. Then $\de=a\de_x$ for some non-zero $a\in K$ (see, e.g., \cite[Lemma~4.1.6]{Stichtenoth:AlgebraicFunctionFieldsAndCodes}).
	We have to show that the absolute differential Galois group of $(K,\de)$ is the free proalgebraic group on a set of cardinality $|K|$. By Lemma \ref{lem: isomorphic absolute differential Galois groups}, it suffices to show that the absolute differential Galois group $H$ of $(K,\de_x)$ is a free proalgebraic group on a set of cardinality $|K|$. Since $(K,\de_x)$ is a finite differential field extension of $(k(x), \frac{d}{dx})$, it follows from Corollary \ref{cor: absolute Galois group is finite index subgroup} that $H$ is a closed finite index subgroup of the absolute differential Galois group of $(k(x), \frac{d}{dx})$. By assumption, the latter is a free proalgebraic group on a set of cardinality $|k(x)|=|K|$. Thus, by Theorem \ref{theo: proalgebraic NielsenSchreier}, $H$ is also a free proalgebraic group on a set of cardinality $|K|$.	 
\end{proof}

From Proposition \ref{prop: Matzat implies Matzat} we immediately obtain Matzat's conjecture for function fields over fields of constants of infinite transcendence degree.

\begin{theo} \label{theo: main matzat}
	Let $k$ be an algebraically closed field of characteristic zero of infinite transcendence degree over $\mathbb{Q}$ and let $K$ be a one-variable function field over $k$, equipped with a non-trivial $k$\=/derivation. Then the absolute differential Galois group of $K$ is the free proalgebraic group on a set of cardinality $|K|$.
\end{theo}
\begin{proof}
	This follows from Proposition \ref{prop: Matzat implies Matzat} because Matzat's conjecture holds for $(k(x),\frac{d}{dx})$, provided that $k$ has infinite transcendence degree, by \cite[Theorem 5.6]{BachmayrHarbaterHartmannWibmer:TheDifferentialGaloisGroupOfRationalFunctionField}.
\end{proof}

\bibliographystyle{alpha}
 \bibliography{bibdata}
\end{document}